\documentclass[11pt]{amsart}

\usepackage{amsmath, amsthm, amssymb} 
\numberwithin{equation}{section}

\newtheorem{theorem}{Theorem}[section]
\newtheorem{proposition}[theorem]{Proposition}

\newtheorem{lemma}[theorem]{Lemma}
\newtheorem{corollary}[theorem]{Corollary}

\newcommand{\Ric}{\operatorname{Ric}}
\newcommand{\iddbar}{\sqrt{-1} \partial \overline{\partial}}
\newcommand{\ddbar}{\partial \overline{\partial}}
\newcommand{\RR}{\mathbb{R}}

\newcommand{\mybar}[1]{\overline{#1}}

\newcommand{\jbar}{{\mybar{j}}}

\newcommand{\qbar}{{\mybar{q}}}

\newcommand{\ov}[1]{\overline{#1}}

\newcommand{\ti}[1]{\tilde{#1}}
\newcommand{\tr}[2]{\mathrm{tr}_{#1} #2}
\newcommand{\osf}{\omega_{\textrm{flat}}}
\newcommand{\ost}{\tilde{\omega}}
\newcommand{\gt}{\tilde{g}}

\newcommand{\real}{\operatorname{Re}}

\title[The continuity equation for Hermitian metrics]{The continuity equation,  Hermitian metrics and elliptic bundles}
\author[M. Sherman]{Morgan Sherman}
\address{Department of Mathematics, California Polytechnic State University, San Luis Obispo, CA
93407}
\author[B. Weinkove]{Ben Weinkove}
\address{Department of Mathematics, Northwestern University, 2033 Sheridan Road, Evanston, IL 60208}

\thanks{Research supported in part by NSF grants DMS-1406164 and DMS-1709544.}

\begin{document}

\maketitle

\begin{abstract}  We extend the continuity equation of La Nave-Tian to Hermitian metrics and establish its interval of maximal existence.   The equation is closely related to the Chern-Ricci flow, and we illustrate this in the case of elliptic bundles over a curve of genus at least two.\end{abstract}

\section{Introduction}

Let $M$ be a compact complex manifold of complex dimension $n$.  Suppose that $M$ admits a K\"ahler metric $\omega_0$.
In \cite{LT}, La Nave-Tian (see also the work of Rubinstein \cite{R}) consider a family of K\"ahler metrics 
$\omega = \omega(s)$ satisfying the \emph{continuity equation} 
\begin{equation}
\label{eqn: ce}
  \omega = \omega_0 - s \Ric (\omega), \quad \textrm{for } s \ge 0.
\end{equation}
Here  
 $\Ric(\omega) = -\iddbar \log \det g$
is the Ricci curvature $(1,1)$ form of $\omega = \sqrt{-1} g_{i\ov{j}}dz^i \wedge d\ov{z}^j$.  This equation was introduced as an alternative to the K\"ahler-Ricci flow in carrying out the Song-Tian analytic minimal model program \cite{ST, ST2}.
The continuity equation has the feature that the Ricci curvature along the path is automatically bounded from below and this has led to several developments \cite{FGS, LTZ, Li, ZZ, ZZ2}.

In this paper we  study a natural analogue of (\ref{eqn: ce}) for non-K\"ahler Hermitian metrics.  If $\omega$ is any Hermitian metric we still define 
\begin{equation}
\label{eqn: ricci curv}
  \Ric (\omega) = -\iddbar \log \det g ,
\end{equation}
which we refer to as the \emph{Chern-Ricci form} of $\omega$.  Unlike the K\"ahler case, in general this form need not 
relate to the full Riemann curvature tensor in any simple fashion.  We now consider the continuity equation (\ref{eqn: ce}) for general Hermitian metrics with the definition (\ref{eqn: ricci curv}).

Our first result establishes the maximal existence interval for the continuity equation.

\begin{theorem}
\label{thm: main}
	Let $M$ be a compact manifold with a Hermitian metric 
	$\omega_0$.  
	Then there exists a unique family of Hermitian metrics 
	$\omega = \omega(s)$ satisfying 
	\begin{equation}
	  \label{eqn: flow}
	  \omega = \omega_0 - s \Ric(\omega),\quad \omega > 0,
	  \quad s \in [0, T),
	\end{equation}
	where $T$ is defined by 
	\begin{equation}
	\label{eqn: maximal time}
	  T := \sup \{ s >0 \mid \exists \psi \in C^{\infty}(M) \emph{ with } \omega_0 - s \Ric(\omega_0) + \iddbar \psi > 0 \} .
	\end{equation}
\end{theorem}

We make some remarks about this result.

\medskip
\noindent
1)  Theorem \ref{thm: main} extends the result of La Nave-Tian \cite{LT} who showed that when $\omega_0$ is K\"ahler, there exists a solution to (\ref{eqn: ce}) up to $T = \sup \{ s >0 \mid [\omega_0] -s c_1(M)>0 \}$, where we are writing $c_1(M) = [\Ric(\omega_0)] \in H^{1,1}(M;\mathbb{R})$ for the first Chern class of $M$ (modulo  a factor of $2\pi$).  This $T$ coincides with the maximal existence time for the K\"ahler-Ricci flow \cite{C, TZ, Ts1, Ts2}.

\medskip
\noindent
2)   The continuity equation (\ref{eqn: flow}) for Hermitian metrics can be regarded as an elliptic version of the \emph{Chern-Ricci flow} 
$$\frac{\partial}{\partial t} \omega = - \Ric(\omega),$$
first introduced by Gill \cite{G}.  Indeed the value $T$ of Theorem \ref{thm: main} coincides with the maximal existence time for the Chern-Ricci flow \cite{TW2}. In particular, if $n=2$ and $\omega_0$ satisfies the Gauduchon condition $\partial \ov{\partial} \omega=0$ then $T$ can be readily computed for many examples (see \cite{TW2, TW3}).

\medskip
\noindent
3)  The value $T$ is independent of the choice of $\omega_0$ in the following sense:  if we replace $\omega_0$ by $\omega_0 + \iddbar f>0$ for a smooth function $f$ then the value $T$ does not change. 

\medskip
\noindent
4) We reduce the proof of Theorem \ref{thm: main} to  an existence result of Cherrier \cite{Ch} (see Theorem \ref{thmch} below).

\medskip

Our second theorem gives an example of the continuity equation (\ref{eqn: flow}) in the setting of elliptic surfaces.  In particular, it will illustrate the close connection to the Chern-Ricci flow.  

Let 
$\pi: M \rightarrow S$ be an  elliptic bundle 
over a Riemann surface $S$ of genus at least $2$.  In particular,  each point $y$ in $S$ has a neighborhood $U$ so that $\pi^{-1}(U)$ is biholomorphic to $U \times T^2$ for a complex 1-dimensional torus $T^2$.  There exist such bundles which are \emph{non-K\"ahler} elliptic surfaces, meaning that they do not admit \emph{any} K\"ahler metric (see  the exposition in \cite[Section 8]{TW3}).  In fact, by the Kodaira classification, \emph{every} minimal non-K\"ahler surface of Kodaira dimension 1 is such an elliptic surface, or admits a finite cover by one (see \cite[Lemmas 1, 2]{B} or \cite[Theorem 7.4]{Wa}).

Denote by $\omega_S$ the unique K\"ahler-Einstein metric on $S$ satisfying $\Ric(\omega_S)=-\omega_S$.  Then the pull-back $[\pi^*\omega_S]$ lies in  $c_1(M)$ and it follows from Theorem \ref{thm: main} that the continuity equation (\ref{eqn: flow}) with any initial $\omega_0$ has $T=\infty$ (see Lemma \ref{lemmamax} below).

Take $\omega_0$ to be a Gauduchon ($\ddbar \omega_0=0$) metric on $M$.  Note that every Hermitian metric is conformal to a Gauduchon one \cite{Ga}.

There exists a family of Gauduchon metrics 
$\omega'(s)$ satisfying the continuity equation
\[
  \omega'(s) = \omega_0 - s \textrm{Ric}(\omega'(s)),
\]
for $s \in [0,\infty)$.  
It is convenient to make a scaling change  (cf. \cite{ZZ})
and consider $\omega(s) = \omega'(s)/(s+1)$ so that the equation becomes
\begin{equation} 
  \label{eqn: eb ce}
  (1+s) \omega(s) 
  = \omega_0 - s \textrm{Ric}(\omega(s)), \quad s \in [0,\infty).
\end{equation}
We call this the \emph{normalized continuity equation}.  Our result describes the behavior of $\omega(s)$ as $s \rightarrow \infty$.

\pagebreak[3]
\begin{theorem}
  \label{thm: eb main}  Let $\pi: M \rightarrow S$ be an elliptic bundle as above, and let $\omega_0$ be a Gauduchon metric on $M$.   Let $\omega(s)$ solve the normalized continuity equation (\ref{eqn: eb ce}).
	As $s \rightarrow \infty$, 
	\begin{enumerate}
	\item[(i)] $\displaystyle{\omega(s) \rightarrow \pi^*\omega_S}$
	in the $C^0(M, \omega_0)$ topology.  
	\item[(ii)] $(M, \omega(s))$ converges to $(S, \omega_S)$ in the Gromov-Hausdorff topology.
	\item[(iii)] The Chern-Ricci curvature of $\omega(s)$ remains uniformly bounded.
	\end{enumerate}
\end{theorem}

The behavior of $\omega(s)$ mirrors the behavior of the Chern-Ricci flow on such elliptic surfaces, which was studied by Tosatti, Yang and the second-named author \cite{TWY}.  Indeed (i) and (ii) hold for both equations, and the proof of Theorem \ref{thm: eb main} makes heavy use of the results and techniques of \cite{TWY}.  A crucial difference is that the Chern-Ricci curvature bound was not obtained in \cite{TWY}, 
suggesting a possible advantage of the continuity equation in this setting.   We also find some simplifications compared to \cite{TWY}.  

Here are some further remarks about Theorem \ref{thm: eb main}.

\medskip
\noindent
1)  We note that the Gauduchon assumption is only used to obtain the identity (\ref{TWY}) below (see \cite[Lemma 3.2]{TWY}) which is used for the bound on the potential $\varphi$ (Lemma \ref{lmm: phi bound} below).

\medskip
\noindent
2) For (i) the precise convergence we obtain is $| \omega(s) - \pi^*\omega_S|_{\omega_0} \le Cs^{-\alpha}$ for any $\alpha \in (0,1/8)$, which corresponds to the exponential convergence for the Chern-Ricci flow in \cite{TWY}.  

\medskip
\noindent
3)  The paper \cite{TWY} considers the metrics restricted to the fibers along the Chern-Ricci flow and obtains convergence (after rescaling) to flat metrics, making use of arguments from \cite{FZ, G2, GTZ, ShW, SW, To}.   The analogous result holds for the continuity equation.  Moreover, the argument of \cite[Theorem 8.2]{TW3} or \cite[Corollary 1.2]{TWY} gives an extension of Theorem \ref{thm: eb main} to all minimal non-K\"ahler elliptic surfaces, by taking a finite cover.   We omit the details   for the sake of brevity and to avoid repetition.

\medskip
\noindent
3) It is not even known if the Chern \emph{scalar} curvature $R$ is uniformly bounded for the Chern-Ricci flow on elliptic bundles.  The bounds for $R$ proved  
 in \cite{TWY}  were $-C \le R \le Ce^{t/2}$, where $t$ is the time parameter along the flow. 

\medskip
\noindent
4) Zhang-Zhang \cite{ZZ} investigated the K\"ahler version of the continuity equation on minimal elliptic K\"ahler surfaces, including those which have non-bundle fibration structures and singular fibers and established the analogue of (ii) (cf. \cite{ST, TZ2}).

\medskip

The outline of the paper is as follows.  In Section \ref{section: pre} we establish notation and state a technical but important lemma for later use.  In Sections \ref{section: pf} and \ref{section: ell} we prove Theorems \ref{thm: main} and \ref{thm: eb main} respectively.

\section{Preliminaries} \label{section: pre}

Given a Hermitian metric $g= (g_{i\ov{j}})$ with associated $(1,1)$ form $\omega = \sqrt{-1} g_{i\ov{j}} dz^i \wedge d\ov{z}^j$ we write $\nabla$ for its Chern connection.  The Christoffel symbols of $\nabla$ are given by $\Gamma^k_{ij} = g^{\ov{q}k} \partial_i g_{j\ov{q}}$, its torsion is given by $T^k_{ij} = \Gamma^k_{ij} -\Gamma^{k}_{ji}$ and the Chern curvature is $R_{k\ov{\ell}i}^{\ \ \ \, p} = - \partial_{\ov{\ell}}\Gamma^p_{ki}$.  We will sometimes raise and lower indices in the usual way using the metric $g$.

The Chern-Ricci curvature of $g$ is the tensor $R_{k\ov{\ell}} = g^{\ov{j}i} R_{k\ov{\ell}i\ov{j}} = - \partial_k \partial_{\ov{\ell}} \log \det g$, and the associated Chern-Ricci form is
$$\Ric(\omega) = \sqrt{-1} R_{k\ov{\ell}} dz^k \wedge d\ov{z}^{\ell},$$
a closed real $(1,1)$ form.  The Chern scalar curvature is the trace $R = g^{\ov{\ell} k} R_{k\ov{\ell}}$.

We use $\Delta$ to denote the complex Laplacian of $g$ which acts on a function $f$ by the formula $\Delta f = g^{\ov{j}i} \partial_i \partial_{\ov{j}} f$.  Given another Hermitian metric $g'$ with associated $(1,1)$ form $\omega'$, we write $\tr{g}{g'} = \tr{\omega}{\omega'} = g^{\ov{j}i} g'_{i\ov{j}}$.

We note here a technical result which  will be useful for later sections.

\begin{proposition} \label{propform}
Let $g=(g_{i\ov{j}})$ and $g'= (g'_{i\ov{j}})$  be Hermitian metrics with $g'_{i\ov{j}} = g_{i\ov{j}} + \partial_i \partial_{\ov{j}} \varphi$, for a smooth function $\varphi$, and define
$$f = \log \frac{\det g'}{\det g}.$$
Then
\[
\begin{split}
\Delta' \log \tr{g}{g'} = {} & \frac{1}{\tr{g}{g'}} \left\{ \frac{2}{\tr{g}{g'}} \emph{Re} \left( g'^{\ov{q}k} T^i_{ik} \nabla_{\ov{q}} \tr{g}{g'} \right) +K + \Delta f - R  \right. \\ 
{} & + g'^{\ov{j}i} \nabla_i \ov{T^{\ell}_{j\ell}} + g'^{\ov{j}i} g^{\ov{\ell}k} g_{p\ov{j}} \nabla_{\ov{\ell}} T^p_{ik} - g'^{\ov{j}i} g^{\ov{\ell}k} g'_{k\ov{q}} (\nabla_i \ov{T^q_{j\ell}} - R_{i\ov{\ell} p \ov{j}} g^{\ov{q}p} ) \\ {} & - \left. g'^{\ov{j}i} g^{\ov{\ell}k} T^p_{ik} \ov{T^q_{j\ell}} g_{p\ov{q}} \right\},
\end{split}
\]
for $K= g^{\ov{\ell}i} g'^{\ov{j}p} g'^{\ov{q}k} B_{i\ov{j}k} \ov{B_{\ell \ov{p}q}} \ge 0$
where
$$B_{i\ov{j}k} = \nabla_i g'_{k\ov{j}} - g'_{i\ov{j}} \frac{\nabla_k \tr{g}{g'}}{\tr{g}{g'}} + T^p_{ik} g'_{p\ov{j}},$$
and $\Delta'$ is the complex Laplacian of $g'$.
\end{proposition}
\begin{proof} This identity is due to Cherrier \cite{Ch}. In this precise form it can be found in  \cite[Section 9]{TW2}.
\end{proof}

\section{Proof of Theorem \ref{thm: main}} \label{section: pf}

In order to prove Theorem \ref{thm: main} we reduce 
the equation (\ref{eqn: flow}) 
to a  complex Monge-Amp\`ere equation on $M$.  
Let $\tilde T \in (0,T)$.  By definition of $T$ there is 
a smooth function $\psi$ such that 
$$\omega_0 - \tilde{T} \Ric(\omega_0) + \iddbar \psi>0.$$
Let $\Omega$ be the volume form given by $\Omega = \omega_0^n e^{\psi/\tilde{T}}$, so that
$$\omega_0 + \tilde{T} \iddbar \log \Omega = \omega_0 - \tilde{T} \Ric(\omega_0) + \iddbar \psi>0.$$
By convexity of the space of Hermitian metrics we also have 
$\omega_0 + s \ddbar \log \Omega > 0$ for each $s \in [0, \tilde T]$.  

\begin{proposition}
\label{lmm: ma} Fix $s \in [0, \tilde{T}]$.  Then there exists a metric  
	 $\omega$ satisfying 
	$\omega = \omega_0 - s \Ric(\omega)$ 
	if and only if 
	there exists a smooth function $u: M \to \RR$ satisfying
	\begin{equation} \label{ma}
	\begin{split}
	  &  \log \frac{ (\omega_0 + s \iddbar \log \Omega + s \iddbar u)^n }{\Omega} - u =  0,  \\
	  & \quad  \omega_0 + s \iddbar \log \Omega + s \iddbar u >  0 . 
	\end{split}
	\end{equation}
\end{proposition}

\begin{proof}  
Suppose first that the metric $\omega = \omega(s)$ satisfies 
$\omega = \omega_0 - s \Ric(\omega)$.  Define $u$ by $u = \log (\omega^n/\Omega)$.  Then $\Ric(\omega) = - \iddbar \log \Omega - \iddbar u$ and so
$$\omega = \omega_0 + s \iddbar \log \Omega + s \iddbar u>0,$$
as required.

Conversely, if $u$ satisfies  (\ref{ma})  
then 
it is straightforward to check that 
$\omega := \omega_0 + s \iddbar \log \Omega + s \iddbar u$ 
satisfies $\omega = \omega_0 - s \Ric(\omega)$.  
\end{proof}

An immediate consequence of the above proposition is the uniqueness of solutions to the continuity equation.

\begin{corollary}
If $\omega'$ and $\omega$ are two metrics solving the continuity equation (\ref{eqn: flow}) for the same $s$ in $[0,T)$ then $\omega'=\omega$.
\end{corollary}
\begin{proof} For $s=0$ there is nothing to prove.  For $s \in (0,T)$, the result follows from uniqueness of solutions $u$ of the equation (\ref{ma}), a  consequence of the maximum principle.
\end{proof}

We now proceed to the proof of Theorem \ref{thm: main}.  First note that (\ref{ma}) is trivially solved when 
$s=0$ by taking $u =  \log \omega_0^n/\Omega$.   Fix $s \in (0,\tilde{T}]$.   Define a new function $\varphi = su$, a Hermitian metric $\hat{\omega}$ by 
$$\hat{\omega} = \omega_0 +  s \iddbar \log \Omega,$$
and a function $F= \log (\Omega/\hat{\omega}^n)$.  Then the equation (\ref{ma}) becomes
\begin{equation*} \label{ma2}
\log \frac{(\hat{\omega} + \iddbar \varphi)^n}{\hat{\omega}^n} = \frac{1}{s} \varphi +F, \quad \hat{\omega} + \iddbar \varphi>0.
\end{equation*}
Recall that $s$ here is fixed.  Then Theorem \ref{thm: main} follows from the following result.

\begin{theorem}[Cherrier \cite{Ch}] \label{thmch}
Let $(M, \hat{\omega})$ be a compact Hermitian manifold,  $F$ a smooth function on $M$ and $\lambda>0$ a constant.  Then there exists a unique solution $\varphi$ to the equation
\begin{equation} \label{ma3}
\log \frac{(\hat{\omega} + \iddbar \varphi)^n}{\hat{\omega}^n} = \lambda \varphi +F, \quad \hat{\omega} + \iddbar \varphi>0.
\end{equation}
\end{theorem}
\begin{proof}
The complex Monge-Amp\`ere equation (\ref{ma3}) is a well-known one in the special case when $\hat{\omega}$ is K\"ahler, and was solved by Aubin \cite{A} and Yau \cite{Y}.  In the Hermitian case, its solution is due to Cherrier \cite{Ch} (note that here $\lambda$ is strictly positive:  for $\lambda =0$ see \cite{Y} and \cite{Ch, TW15}).  For the sake of completeness, we include here a brief sketch of the proof.

We introduce a parameter $t \in [0,1]$ and consider the family of equations
\begin{equation} \label{ma4}
\log \frac{(\hat{\omega} + \iddbar \varphi)^n}{\hat{\omega}^n} = \lambda \varphi +tF, \quad \hat{\omega} + \iddbar \varphi>0.
\end{equation}
for $\varphi =\varphi(t)$.  Let $E$ denote the set of those $t \in [0,1]$ for which (\ref{ma4}) has a solution.  Note that $0 \in E$ since $\varphi=0$ is trivially a solution.  It suffices to show that $E$ is both open and closed.

For the openness of $E$, fix $\alpha \in (0,1)$ and consider the map
$$\Psi : [0,1] \times C^{2,\alpha}(M) \rightarrow C^{\alpha}(M), \quad \Psi(t, \varphi) = \log \frac{(\hat{\omega} + \iddbar \varphi)^n}{\hat{\omega}^n} - \lambda \varphi -tF.$$
Assume $t_0 \in E$, and that (\ref{ma4}) has a corresponding solution $\varphi_0$.  Write $\omega_0 = \hat{\omega} + \iddbar \varphi_0$ and $g_0$ for the corresponding Hermitian metric.  The derivative of  $\Psi$ in the second variable at $(t_0, \varphi_0)$ is the linear operator
$L: C^{2,\alpha}(M) \rightarrow C^{\alpha}(M)$ given by
$$L f = \Delta_0 f - \lambda f,$$
for $\Delta_0$ the Laplacian of $\omega_0$. 
The maximum principle implies that $L$ is injective.   By the Implicit Function Theorem, the surjectivity of $L$ is sufficient to show the openness of $E$.
Following an argument similar to that of \cite{TW1}, we compute the $L^2$ adjoint of this operator with respect to a specific volume form on $M$, making use of a theorem of Gauduchon \cite{Ga}.  Let $\sigma$ be a smooth function such that $\omega_G := e^{\sigma} \omega_0$ is a Gauduchon metric, namely that $\ddbar \omega_G^{n-1}=0$.  Then compute for a smooth function $h$,
$$\int_M (\Delta_0 f) h e^{(n-1)\sigma} \omega_0^n = \int_M f \left( \Delta_0 h + 2n \textrm{Re} \left( \frac{\sqrt{-1} \partial h \wedge \ov{\partial} \omega_G^{n-1}}{e^{(n-1)\sigma} \omega_0^n}  \right) \right) e^{(n-1)\sigma}\omega_0^n.$$
Hence the adjoint of $L$ with respect to $e^{(n-1)\sigma}\omega_0^n$ is given by
$$L^* h = \Delta_0 h + 2n \textrm{Re} \left( \frac{\sqrt{-1} \partial h \wedge \ov{\partial} \omega_G^{n-1}}{e^{(n-1)\sigma} \omega_0^n} \right) - \lambda h,$$
and the maximum principle implies that $L^*$ is injective.  By the Fredholm alternative, $L$ is surjective.

For the closedness of $E$ we need \emph{a priori} estimates on $\varphi$ solving (\ref{ma4}), independent of $t$.  A uniform bound $|\varphi| \le C$ follows immediately from the maximum principle.  Here and henceforth, $C$ will denote a uniform constant that may change from line to line.

Write $\omega' = \hat{\omega}+ \iddbar \varphi$, and let $g'$ be the associated Hermitian metric.  We will bound $\tr{\hat{g}}{g'}$ from above.  By the bound on $|\varphi|$, the equation (\ref{ma4}) and the arithmetic-geometric means inequality this will imply the uniform equivalence of the metrics $\hat{g}$ and $g'$.  We follow the argument of \cite[Section 9]{TW2} which uses a trick of Phong-Sturm \cite{PS} and consider the quantity
$$Q = \log \tr{\hat{g}}{g'} - A \varphi + \frac{1}{\varphi - \inf_M \varphi +1},$$
for $A$ a constant to be determined.
As in (9.4) of \cite{TW2},
\begin{equation} \label{DQ}
\Delta' Q \ge \Delta' \log \tr{\hat{g}}{g'} + A \tr{g'}{\hat{g}} + \frac{2|\partial \varphi|^2_{g'}}{(\varphi - \inf_M \varphi+1)^3} - An -n.
\end{equation}
Next we apply Proposition \ref{propform} with $f = \lambda \varphi + tF$ to obtain
\begin{equation} \label{Dp}
\Delta' \log \tr{\hat{g}}{g'} \ge \frac{2}{(\tr{\hat{g}}{g'})^2} \textrm{Re} (g'^{\ov{q}k} \hat{T}^i_{ik} \hat{\nabla}_{\ov{q}} \tr{\hat{g}}{g'}) -C \tr{g'}{\hat{g}} -C,
\end{equation}
noting that $\hat{\Delta} f \ge -C$ since $\hat{\Delta} \varphi > -n$.

We compute at a point $x_0 \in M$ at which $Q$ achieves its maximum.   At $x_0$ we have  $\partial Q=0$ and, assuming without loss of generality that $\tr{\hat{g}}{g'}$ is large compared to $A$, we obtain
\begin{equation} \label{useful}
\left| \frac{2}{(\tr{\hat{g}}{g'})^2} \textrm{Re} (g'^{\ov{q}k} \hat{T}^i_{ik} \hat{\nabla}_{\ov{q}} \tr{\hat{g}}{g'}) \right| \le \frac{ |\partial \varphi|^2_{g'}}{(\varphi - \inf_M \varphi+1)^3} + C \tr{g'}{\hat{g}},
\end{equation}
recalling that $| \varphi| \le C$.  Combining (\ref{DQ}), (\ref{Dp}) and (\ref{useful}) and choosing $A$ sufficiently uniformly large, we obtain that $\tr{g'}{\hat{g}}\le C$ at $x_0$, and an upper bound for $Q$ follows.  This implies an upper bound for $\tr{\hat{g}}{g'}$ on $M$ and hence the uniform ellipticity of the equation (\ref{ma4}).  Then $C^{2,\alpha}$ estimates for $\varphi$ follow from the Evans-Krylov theory \cite{E,K, Tr} or \cite{TWWY}, and higher order estimates for $\varphi$ follow from a standard bootstrap procedure.
\end{proof}

\section{Elliptic bundles} \label{section: ell}

In this section we give a proof of Theorem \ref{thm: eb main}.
As in the introduction, let
$\pi: M \rightarrow S$ be an elliptic bundle 
over a Riemann surface $S$ of genus at least $2$, 
and let $\omega_0$ be a Gauduchon metric on $M$.

We follow the notation used in \cite{TWY}, 
and use several important facts established there.  
For convenience we restate the relevant facts here and 
refer the reader to the paper for further details.  
Note that when comparing our notation here with that 
in \cite{TWY}, our quantity $s$ relates to the quantity $t$ 
in that paper by the equation $1+s = e^t$.  

Given $y \in S$ we denote by $E_y := \pi^{-1}(y)$ 
the fiber over $y$, which by assumption is 
isomorphic to a torus.  
There is a smooth function $\rho : M \to \RR$ such that 
the form
\[
  \osf := \omega_0 + \iddbar \rho  
\]
has the property that its restriction to each fiber 
$E_y$ is the unique flat metric on $E_y$ in the cohomology class 
of $\omega_0 |_{E_y}$.  
We refer to $\osf$ as the \emph{semi-flat form}.  
It is not necessarily a metric since it may not be 
positive definite on $M$.

We denote by $\omega_S$ the pullback $\pi^* \omega_S$ of the unique 
K\"ahler-Einstein metric on $S$.  
This form lies in $-c_1(M)$;    
fix the volume form $\Omega$ which satisfies 
$\iddbar \log \Omega = \omega_S$ and 
$\int_M \Omega = 2 \int_M \omega_0 \wedge \omega_S$. By \cite[Lemma 3.2]{TWY} we have
\begin{equation} \label{TWY}
\Omega = 2 \osf \wedge \omega_S.
\end{equation}  

From the equation $\iddbar \log \Omega = \omega_S$, every Hermitian metric has Chern-Ricci form equal to $-\omega_S + \iddbar \psi$ for some function $\psi$.  Since $\omega_S \ge 0$,   Theorem \ref{thm: main} immediately implies  the following (cf. \cite[Theorem 1.5]{TW2}).

\begin{lemma} \label{lemmamax}
The maximal existence interval for the continuity equation with any initial metric on $M$ is $[0,\infty)$.
\end{lemma}

We use the \emph{normalized continuity equation} (\ref{eqn: eb ce}), namely
 $\omega=\omega(s)$ is a family of Gauduchon metrics solving: 
\begin{equation} 
  \label{ce2}
  (1+s) \omega 
  = \omega_0 - s \textrm{Ric}(\omega), \quad s \in [0,\infty).
\end{equation}

We set 
\begin{equation}
  \label{eqn: omega tilde defn}
  \ost = \ost(s) = \frac{1}{1+s} \osf 
  + \frac{s}{1+s} \omega_S.
\end{equation}
Note that $\ost$ may not be positive definite for every $s > 0$, 
but that it will be for every $s$ sufficiently large.  
We will use $\ost(s)$ as a path of reference metrics to 
reduce (\ref{ce2}) to a complex Monge-Amp\`ere equation.    We claim that (\ref{ce2}) is equivalent to
\begin{equation} \label{mae}
\log \frac{(1+s)(\tilde{\omega}+\iddbar \varphi)^2}{\Omega} = \frac{1+s}{s} \varphi + \frac{1}{s} \rho, \quad \omega = \tilde{\omega}+\iddbar \varphi>0
\end{equation}
Indeed, if $\omega$ solves (\ref{ce2}) then define $\varphi$ by
$$(1+s)\varphi = s \log \left(\frac{ (1+s)\omega^2}{\Omega} \right) -\rho,$$
and then applying $\iddbar$ to both sides and rearranging we obtain
$$(1+s) \left( \tilde{\omega} + \iddbar \varphi \right) = \omega_0 - s \Ric(\omega),$$
from which it follows that $\omega = \tilde{\omega}+\iddbar \varphi$.   Likewise,  if $\varphi$ solves (\ref{mae}) then $\omega:= \tilde{\omega}+\iddbar \varphi$ solves (\ref{ce2}).

We now turn to the proof of Theorem \ref{thm: eb main}, by establishing uniform estimates for $\omega=\omega(s)$.  Note that we may assume without loss of generality that $s$ is sufficiently large so that $\ost(s)$ is positive definite.

We begin with:

\begin{lemma}
  \label{lmm: phi bound}
  There is a uniform constant $C$ such that 
  \[
    |\varphi| \le \frac{C}{1+s} . 
  \]
\end{lemma}
\begin{proof} We follow  \cite[Lemma 6.7]{SW} and \cite[Lemma 3.4]{TWY}. 
Since  
$\Omega = 2 \osf \wedge \omega_S$ 
we have:
\begin{equation} \label{ue}
\begin{split}
   \frac{(1+s)\ost^2}{\Omega} 
  =  {} & \frac{\osf^2 + 2s \osf \wedge \omega_S}{(1+s)\Omega} \\
  =  {} & 1 + \frac1{1+s}\left(\frac{\osf^2}{\Omega}-1\right)  \\ = {} & 1+O(1/s)
\end{split}
\end{equation}
Since $s \log (1+O(1/s))$ is bounded as $s \rightarrow \infty$ we obtain
\begin{equation} \label{elog}
  \left| s \log \frac{(1+s)\ost^2}{\Omega} \right| \le C.
\end{equation}
Now we apply the maximum principle.  
Suppose $\varphi$ achieves its maximum at a point $x_0$.  
Then at $x_0$ we have $\iddbar \varphi \le 0$ and hence $\omega \le \tilde{\omega}$ and $\omega^2 \le \tilde{\omega}^2$.  Then by (\ref{mae}), at $x_0$,
$$\varphi \le - \frac{\rho}{1+s} + \frac{s}{1+s} \log \frac{(1+s)\tilde{\omega}^2}{\Omega} \le \frac{C}{1+s},$$
by (\ref{elog}), giving the upper bound for $\varphi$.  The lower bound is similar.
\end{proof}

Next we show that the volume forms of $\omega$ and $\ost$ are  uniformly equivalent, and in fact approach each other as $s \rightarrow \infty$.

\begin{lemma}
  \label{lemma:  det omega bound}
  There is a uniform constant $C > 0$ such that for $s$ sufficiently large,
  \[
   \ost^2\left( 1- \frac{C}{s} \right) \le \omega^2 \le  \ost^2 \left(1+\frac{C}{s} \right).
  \]
\end{lemma}

\begin{proof}
From (\ref{ue})
we see that for $s$ sufficiently large
$$\frac{\omega^2}{\tilde{\omega}^2} = \frac{(1+s)\omega^2}{\Omega} \frac{\Omega}{(1+s)\tilde{\omega}^2} = \frac{(1+s)\omega^2}{\Omega} (1+O(1/s)).$$
But from (\ref{mae}) and Lemma \ref{lmm: phi bound}
$$\frac{(1+s)\omega^2}{\Omega} = \exp\left( \frac{1+s}{s} \varphi + \frac{\rho}{s} \right) = 1+O(1/s),$$
and the result follows.
\end{proof}

We now turn to proving that $\tr{\gt}{g}$ is uniformly bounded, where $g, \ti{g}$ are the Hermitian metrics associated to $\omega, \ti{\omega}$.
This, together with the previous lemma will show that $\omega$ and $\ost$ are 
uniformly equivalent.  
We denote by $\tilde{\nabla}$ the Chern connection of $\gt$.  
Similarly we will write $\tilde{T}_{ik}^p$ and $\tilde{R}_{i \jbar k \ov{\ell}}$ 
 for the  torsion  and curvature tensors 
of $\gt$, and $|\tilde{T}|_{\tilde{g}}$, $|\widetilde{\textrm{Rm}}|_{\tilde{g}}$ for their norms with respect to $\tilde{g}$.  We begin with a technical lemma from \cite{TWY}.

\begin{lemma} \label{lemmaTR}
For $s$ sufficiently large,
$$
   | \tilde{T} |_{\gt} \le C, 
  \quad  | \overline{\tilde{\nabla}} \tilde{T} |_{\gt} 
    + | \tilde{\nabla} \tilde{T} |_{\gt} 
    + | \widetilde{\mathrm{Rm}} |_{\gt} \le C \sqrt{s}, 
$$
for a uniform constant $C$.
\end{lemma}
\begin{proof} See \cite[Lemma 4.1]{TWY}.
\end{proof}

In order to apply the maximum principle to $\tr{\gt}{g}$ we will use the following lemma (cf. \cite[Lemma 5.2]{TWY}).

\begin{lemma}
  \label{lmm: laplacian log trace omega}
	For $s$ sufficiently large,
	\[
	  \Delta \log \tr{\gt}{g} 
	  \ge 
	   \frac{2}{(\tr{\gt}{g})^2} 
	  \real \left( g^{\qbar k} \tilde{T}_{ik}^i 
	  \tilde{\nabla}_{\qbar} \tr{\gt}{g} \right) 
	  - C \sqrt{s}  \, \tr{g}{\gt}.
	\]
\end{lemma}

\begin{proof}  Write (\ref{mae}) as
$$\log\frac{(\tilde{\omega}+\iddbar \varphi)^2}{\tilde{\omega}^2} = f,$$
for 
$$f = \frac{1+s}{s} \varphi + \frac{1}{s} \rho + \log \frac{\Omega}{(1+s)\tilde{\omega}^2}.$$
Compute
\[
\begin{split}
\tilde{\Delta} f = {} & \frac{1+s}{s} \tilde{\Delta} \varphi + \frac{1}{s} \tilde{\Delta} \rho +  \tr{\tilde{\omega}}\omega_S + \tilde{R} \\
= {} & \frac{1+s}{s} ( \tr{\tilde{\omega}}{\omega} - 2) +  \frac{2(1+s)}{s} -  \frac{1}{s} \tr{\tilde{\omega}}{\omega_0}  + \tilde{R},
\end{split}
\]
where we used $\iddbar \rho =  (1+s) \tilde{\omega} - s \omega_S - \omega_0$.

Applying Proposition \ref{propform} we obtain
\[
\begin{split}
\lefteqn{\Delta \log \tr{\ti{g}}{g}} \\ \ge {} & \frac{1}{\tr{\gt}{g}} \left\{ \frac{2}{\tr{\gt}{g}} \textrm{Re} \left( g^{\ov{q}k} \ti{T}^i_{ik} \ti{\nabla}_{\ov{q}} \tr{\gt}{g} \right)  -C - \frac{1}{s} \tr{\tilde{\omega}}{\omega_0}  
 + g^{\ov{j}i} \ti{\nabla}_i \ov{\ti{T}^{\ell}_{j\ell}} \right. \\ {} & + \left. g^{\ov{j}i} \ti{g}^{\ov{\ell}k} \ti{g}_{p\ov{j}} \ti{\nabla}_{\ov{\ell}} \ti{T}^p_{ik} - g^{\ov{j}i} \ti{g}^{\ov{\ell}k} g_{k\ov{q}} (\ti{\nabla}_i \ov{\ti{T}^q_{j\ell}} - \ti{R}_{i\ov{\ell} p \ov{j}} \ti{g}^{\ov{q}p} ) - g^{\ov{j}i} \ti{g}^{\ov{\ell}k} \ti{T}^p_{ik} \ov{\ti{T}^q_{j\ell}} \ti{g}_{p\ov{q}} \right\},
\end{split}
\]
for a uniform constant $C$.  From the definition of $\ost$ we see that for $s$ sufficiently large, $\omega_0 \le C s \tilde{\omega}$ and hence 
 $\frac{1}{s} \tr{\tilde{\omega}}{\omega_0} \le C$. Note that by Lemma \ref{lemma:  det omega bound}, $\tr{g}{\tilde{g}} \ge c$ for a uniform $c>0$ and $\tr{g}{\tilde{g}}$ is uniformly equivalent to $\tr{\tilde{g}}{g}$.
Applying Lemma \ref{lemmaTR}  completes the proof.
\end{proof}

We can now obtain the bound on 
$\tr{\gt}{g}$.  

\begin{lemma}
  \label{lem: eb tr gtilde g bd}
  We have 
  \[
    \tr{\gt}{g} \le C,
  \]
  and hence $\omega$ and $\ost$ are uniformly equivalent.  
\end{lemma}

\begin{proof}
As in \cite[Theorem 5.1]{TWY}, define
\[
  Q = \log \tr{\tilde{g}}{g} 
  - A \sqrt{s}\,  \varphi 
  + \frac{1}{\tilde{C}+ \sqrt{s} \, \varphi}
\] 
where $\tilde{C}$ is chosen so that 
$\tilde{C}+ \sqrt{s}\, \varphi \ge 1$ 
(see Lemma \ref{lmm: phi bound}).  
Now $\Delta \varphi = 2 -\tr{g}{\tilde{g}}$ and so
\[
  \begin{split}
  \Delta \left(   - A \sqrt{s}\,  \varphi 
  + \frac{1}{\tilde{C}+ \sqrt{s} \, \varphi} \right) 
  = {} &  \left( -A - \frac{1}{(\tilde{C} + \sqrt{s} \, \varphi)^2} \right) 
  \Delta (\sqrt{s} \, \varphi) 
  + \frac{2 | \partial (\sqrt{s}\, \varphi ) |^2_g}%
         {(\tilde{C} + \sqrt{s} \, \varphi )^3} 
  \\
  {} &  \ge - C A \sqrt{s} + A \sqrt{s} \, \tr{g}{\tilde{g}} 
  + \frac{2 | \partial (\sqrt{s}\, \varphi ) |^2_g}%
         {(\tilde{C} + \sqrt{s} \, \varphi )^3}.
  \end{split}
\]
At the point $x_0$ where $Q$ achieves a maximum we have 
$\partial_{\ov{q}}Q=0$ so that
\begin{align*}
  \frac{2}{(\tr{\gt}{g})^2} \real \left( 
    g^{\qbar k} \tilde{T}_{ik}^i \tilde{\nabla}_{\qbar} \tr{\gt}{g} \right) 
    ={} & \frac{2}{\tr{\gt}{g}} \real \left( 
    g^{\qbar k} \tilde{T}_{ik}^i (A + \frac1{(\tilde{C} + \sqrt{s} \varphi)^2}) 
    \sqrt{s} \partial_{\qbar} \varphi \right) \\
  \ge {} &  - \frac{4A}{\tr{\gt}{g}} 
    \left| g^{\qbar k} \tilde{T}_{ik}^i \right|_g
    \left| \partial_{\qbar} (\sqrt{s} \varphi) \right|_g \\
  \ge {} & - \frac{CA^2}{(\tr{\gt}{g})^2} 
    \left| g^{\qbar k} \tilde{T}_{ik}^i \right|_g^2 
    (\tilde{C} + \sqrt{s} \varphi)^3
    - \frac{\left| \partial (\sqrt{s} \varphi) \right|_g^2}%
           {(\tilde{C} + \sqrt{s} \varphi)^3} \\
  \ge  {} & - \frac{C' A^2}{\tr{\gt}{g}} 
    - \frac{\left| \partial (\sqrt{s} \varphi) \right|_g^2}%
           {(\tilde{C} + \sqrt{s} \varphi)^3},
\end{align*}
where we have used Lemmas \ref{lmm: phi bound} and  \ref{lemmaTR}.
Then, at $x_0$, from Lemma \ref{lmm: laplacian log trace omega},
\begin{align*}
  0 &\ge \Delta Q 
  \ge - C A^2 + (A-C) \sqrt{s} \, \tr{g}{\gt} - CA \sqrt{s}
  \ge - CA^2 + \sqrt{s} \, \tr{g}{\gt} - CA \sqrt{s} 
\end{align*}
if we choose $A \ge C+1$.  
So at this point $\tr{g}{\gt}$, and hence $\tr{\gt}{g}$ is bounded from above and the result follows.
\end{proof}

Next we show that $g$ and $\ti{g}$ approach each other as $s \rightarrow \infty$.

\begin{lemma} \label{lemmaconv}
For every $\alpha$ with $0<\alpha<1/4$ there is a constant $C$ such that for $s$ sufficiently large,
\begin{enumerate}
\item[(a)] $\displaystyle{\tr{\gt}{g}-2\le Cs^{-\alpha}}.$
\item[(b)] $\displaystyle{\tr{g}{\gt} -2 \le C s^{-\alpha}}.$
\item[(c)] $\displaystyle{(1- Cs^{-\alpha/2})\tilde{g} \le g \le (1+Cs^{-\alpha/2}) \ti{g}.}$ 
\end{enumerate}
\end{lemma}
\begin{proof}  We use the idea from \cite[Proposition 7.3]{TWY}, but in our case the argument is slightly easier.
Now that $g$ and $\ti{g}$ are uniformly equivalent, it follows from Lemma \ref{lmm: laplacian log trace omega}
that 
$$\Delta \tr{\gt}{g}  = \tr{\gt}{g} \Delta \log \tr{\gt}{g} + \frac{ |\nabla \tr{\gt}{g}|^2_g}{\tr{\gt}{g}} \ge - C \sqrt{s},$$ where we always assume $s$ is sufficiently large. Define $\beta = 1/2 + 2\alpha <1$ and $Q = s^{\alpha} (\tr{\tilde{g}}{g} -2) - s^{\beta}\varphi$.  Compute
$$\Delta Q \ge - Cs^{\alpha+1/2} + s^{\beta}(\tr{g}{\ti{g}}-2) \ge -C s^{\alpha+1/2} + s^{\beta}(\tr{\ti{g}}{g} -2) - C' s^{\beta-1},$$
where for the last inequality we used
\begin{equation} \label{trdet}
\tr{g}{\ti{g}} = \tr{\ti{g}}{g} + \left( \frac{\det \ti{g}}{\det g}-1 \right) \tr{\ti{g}}{g} = \tr{\ti{g}}{g} + O(1/s),
\end{equation}
which follows from Lemma \ref{lemma:  det omega bound}.

Hence at the point where $Q$ achieves a maximum,
$$s^{\alpha}(\tr{\ti{g}}{g} -2) \le  Cs^{2\alpha +1/2-\beta} +Cs^{\alpha-1} \le 2C.$$
But from Lemma \ref{lmm: phi bound}, $s^{\beta} |\varphi|$ is bounded, and hence $Q$ is bounded, giving (a).  

Part (b) follows from (a) and (\ref{trdet}).  Part (c) is an elementary consequence of parts (a) and (b) (see \cite[Lemma 7.4]{TWY}).
\end{proof}

Now part (i) of Theorem \ref{thm: eb main} follows from part (c) of this lemma and the definition of $\tilde{\omega}$.  Part (ii) is a consequence of (i) (see \cite[Lemma 9.1]{TWY}).  The next result completes the proof of Theorem \ref{thm: eb main}.

\begin{lemma}
  \label{thm: chern scalar curv bd}
  For $\Ric(\omega)$ the Chern-Ricci curvature of $\omega=\omega(s)$, we have 
  \[
    - C\omega \le \Ric(\omega) \le C\omega
  \]
  for a uniform constant $C$. 
\end{lemma}

\begin{proof}
From the continuity equation 
(\ref{ce2})
we have 
\[
  \Ric(\omega) = \frac1s \omega_0 - \frac{(1+s)}{s} \omega.
\]
Hence $\Ric(\omega) \ge - \frac{(1+s)}{s} \omega$ 
giving immediately the lower bound of $\Ric(\omega)$.

For the upper bound we have for $s$ sufficiently large,
\[
  \Ric(\omega)
  \le \frac{1}{s} \omega_0 \le C \omega,
\]
since $\omega_0 \le Cs \tilde{\omega}$  and 
$\tilde{\omega}$ and $\omega$ are equivalent.
\end{proof}

\end{document}